\documentclass{amsart}
\usepackage{amssymb,amsmath,amsthm,epsf,epsfig,dsfont,bbm}

\usepackage{latexsym}
\usepackage{upref, eucal}
\usepackage[all]{xy}

\newcommand {\nc} {\newcommand}
\newcommand {\enm} {\ensuremath}

\nc {\bdm} {\begin{displaymath}}
\nc {\edm} {\end{displaymath}}

\newtheorem {theorem} {\bf{Theorem}}[section]
\newtheorem {lemma}[theorem] {\bf Lemma}
\newtheorem {proposition}[theorem] {\bf Proposition}

\newtheorem {corollary}[theorem] {\bf Corollary}
\numberwithin {equation}{section}

\newcommand{\Ou}{\enm{\mathcal{O}}}

\nc{\J}{\enm{\mathcal{J} }}
\nc {\Z} {\enm{\mathbb{Z}}}
\nc {\form}[1] {\enm{\mbox{\underline{for}}}_{#1}}
\nc {\prol}[1] {\enm{\mbox{\underline{prol}}_{{#1}^*}}}
\newcommand{\Del}{\partial}
\nc {\stk} {\stackrel}

\newcommand{\map}{\rightarrow}

\newcommand{\beqar}{\begin{eqnarray*}}
\newcommand{\eeqar}{\end{eqnarray*}}
\newcommand{\inj}{\hookrightarrow}

\newcommand{\Pn}[2] {\ensuremath{ {\mathbb{P}}^{#1}_{#2}}}
\nc{\Quot}[3]{\enm{ {\mathfrak{Quot}_{ {#1}/{#2}/{#3}}}}}
\nc{\Hilb}[2]{\enm{ {\mathfrak{Hilb}_{ {#1}/{#2}}}}}

\newcommand{\bb}[1]{\mathbb{#1}}
\newcommand{\mcal}[1]{\mathcal{#1}}

\nc {\Coh}[4] {\ensuremath{H^{#1}(\Pn{#2}{},{#3}({#4}))}}
\nc {\Ch}[3] {\enm{H^{#1}(X_t,{#2}_t({#3}))}}
\nc {\Qphi}[4]{\enm{ {\mathfrak{Quot}^{~#4}_{ {#1}/{#2}/{#3}}}}}
\nc {\Gra}[4]{\enm{ {\mathfrak{Grass}_{#2}({#3},{#4})}}}
\nc {\HomA}[2]{\enm{\mathrm{Hom}_A{#1}{#2}}}
\nc {\tr}{\mathrm{tr}}

\nc {\C}[2]{\enm{\left(\begin{array}{l} {#1} \\ {#2} \end{array} \right)}}
\nc {\mat}[4]{\enm{\left(\begin{array}{ll}{#1} & {#2} \\ {#3} & {#4}
\end{array}\right)}}

\def \mb{\mbox}

 \def \Z{{\mathbb Z}}

   \def \h{\hat{\ }}

  \def \bX{{\bf X}} \def \bH{{\bf H}}

\def \hG{\hat{\mathbb{G}}_{\mathrm{a}}} 
\def \Ga{\mathbb{G}_{\mathrm{a}}}

\def \R1{R((q))[q']\h}

\usepackage{xcolor}

\DeclareMathOperator{\Spec}{\mathrm{Spec}}

\DeclareMathOperator{\rk}{\mathrm{rk}}
\newcommand{\Hom}{\mathrm{Hom}}

\newcommand{\Ext}{\mathrm{Ext}}

\newcommand{\bI}{{\bf I}}

\newcommand{\oldmarginpar}[1]{}

\usepackage{verbatim}

\newcommand{\cy}{}
\nc{\bx}{\bold{x}}
\nc{\by}{\bold{y}}
\nc{\bz}{\bold{z}}
\nc{\ba}{\bold{a}}
\nc{\Fd}{F^\dagger}
\nc{\Rd}{R^\dagger}
\nc{\mlow}{m_{\mathrm{l}}}
\nc{\mup}{m_{\mathrm{u}}}
\nc{\ord}{\mb{ord }}
\nc{\bXp}{\bX_{\mathrm{prim}}}
\nc{\bPsi}{\bold{\Psi}}
\nc{\mult}{\mathrm{mult}}
\nc{\mbB}{\mathbbm{B}}
\nc{\mfor}[1]{{#1}^{\mathrm{for}}}
\nc{\latder}{\tilde{\partial}}
\nc{\bg}{{\bold h}}
\nc{\bff}{{\bold f}}
\nc{\bu}{{\bold u}}
\nc{\tTheta}{\tilde{\Theta}}
\nc{\Ker}{\text{Ker}}
\nc{\tlu}{\tilde{u}}
\nc{\cblue}{\color{black}}
\nc{\cblack}{\color{black}}

\nc{\cblu}[1]{\color{black}{#1} \color{black}}

\nc{\cment}[1]{}

\nc{\Ha}{\bb{A}^g}

\nc{\bHnew}{\tilde{\bH}}
\usepackage [english]{babel}
\usepackage [autostyle, english = american]{csquotes}
\MakeOuterQuote{"}

\title{Differential Characters and $D$-Group Schemes}
\author{Rajat Kumar Mishra and Arnab Saha}
\date{}
\email{rajat.m@iitgn.ac.in, arnab.saha@iitgn.ac.in}
\address{Indian Institute of Technology Gandhinagar, Gujarat 382355, India}

\subjclass[2010]{Primary 14L15, 14L40, 14K99}

\keywords{differential algebra, differential characters, group schemes, 
jet spaces}

\usepackage{hyperref}
\begin{document}
\maketitle

\begin{abstract}
\cblue
Let $K$ be a field of characteristic zero with a fixed derivation $\partial$ 
on it.
In the case when $A$ is an abelian scheme, Buium considered the 
group scheme $K(A)$ which is the kernel of differential characters (also known
as Manin characters) on the 
jet space of $A$. Then $K(A)$ naturally inherits a $D$-group scheme structure.
Using the
theory of universal vectorial extensions of $A$, he further showed that $K(A)$ 
is a finite dimensional vectorial extension of $A$.

Let $G$ be a smooth connected commutative finite dimensional group 
scheme over $\Spec K$. 
In this paper, using the  theory of differential characters, 
we show that the associated kernel group scheme $K(G)$ is a finite 
dimensional $D$-group scheme 
that is a vectorial extension of such a general $G$.

Our proof relies entirely on understanding the structure of jet spaces. 
Our method also allows us to
give a classification of the module of differential characters 
$\mathbf{X}_\infty(G)$  in terms of primitive characters as a 
$K\{\partial\}$-module.

\cblack



\end{abstract}

\color{black}
\section{Introduction}
The theory of differential algebra was initiated by Ritt \cite{Ri} to study the geometry of algebraic differential equations. Since then the theory has seen 
remarkable developments \cite{Bert}\cite{B0}\cite{Cassidy}\cite{Hru}\cite{K1} 
and applications in diophantine geometry, some of which may be found in 
\cite{B4}\cite{BuPi}\cite{BV1}\cite{Hru}\cite{Manin}\cite{V1}. 
Buium developed several aspects of the theory for curves, abelian varieties and their moduli spaces in a series of papers \cite{B1}
\cite{B4}\cite{B8}. 

Differential characters play a central role in this study and its applications to diophantine geometry. 
\cblue 
Let $K$ be a field of characteristic zero with a fixed derivation $\partial$ 
on it.
In the case when $A$ is an abelian scheme over $\Spec K$, Buium considered the 
kernel group scheme $K(A)$ of differential characters on $A$.
By definition, $K(A)$ is a closed subscheme of the infinite jet space 
$J^\infty A$ and naturally inherits a $D$-group scheme
structure by restricting the one on $J^\infty A$. 
Note that $J^\infty A$ is an infinite dimensional $D$-group scheme over 
$\Spec K$. In Section $2.1$ of \cite{B4} or, Section $5$ of \cite{Bu1}, 
Buium used the theory of universal vectorial extensions of $A$ to 
show that $K(A)$ is a finite dimensional vectorial extension of $A$.


In this paper, for any smooth connected commutative finite dimensional group
scheme $G$ over $\Spec K$, using the theory of jet schemes, we 
show that the associated kernel group scheme $K(G)$ is a finite dimensional
$D$-group scheme and  is a vectorial
extension of $G$; the precise structure of $K(G)$ is given in our Theorem 
\ref{K(G)}.

The module of differential characters (also known as Manin characters) of $G$, denoted by $\mathbf{X}_\infty(G)$, naturally has a $K\{\partial\}$-action on it, where $K\{\partial\}$ is the non-commutative ring generated by the derivative operator $\partial$ over $K$. One then naturally considers the subclass of differential characters of the group scheme which are not obtained by any $K$-linear combination of compositions of the operator $\partial$ on lower order characters. The elements of this subclass of differential characters are called primitive characters.

In this paper, we prove that $\mathbf{X}_\infty(G)$ is finitely
generated over $K\{\partial\}$ by $g$ primitive differential characters where $g$ is the relative dimension of $G$ over $\Spec K$. We also give an upper bound on the order of the primitive 
characters generating the module of differential characters over 
$K\{\partial\}$ in terms of the dimension of the Ext group of $G$. The 
precise statement of the above discussion is our Theorem \ref{mt}.
The analogous results in the case of arithmetic jet spaces are 
considered in \cite{BoSa1}\cite{BoSa2}\cite{Bu2}\cite{Bu-Mi-1}\cite{Bu-Mi-2}\cite{PS1}.
\cblack


Now we explain our results in greater detail. Let $X$ be a scheme over 
$\Spec K$. 
Then the system of jet spaces $J^*X=\{J^nX\}_{n=0}^\infty$ with respect to $(K,\partial)$ forms an inverse system of schemes, called the canonical prolongation sequence. This system satisfies a universal property with respect to lifting the fixed derivation $\partial$ on $K$. The ring of global functions $\Ou(J^nX)$ for
each $n$ is the ring of the $n$-th order differential functions on $X$.

For any $K$-algebra $B$, let  $D_n(B):=\frac{B[\epsilon]}{(\epsilon^{n+1})}$ be the ring of truncated polynomials of length $n+1$ with coefficients in $B$. 
\cblu{
Composing the structure map along with the canonical Hasse-Schmidt 
derivation $\exp_\partial: K \map D_n(K)$, $D_n(B)$ naturally becomes a 
$K$-algebra. 
Note that when $\partial$ is a non-trivial
operator on $K$, then the $K$-algebra structure on $D_n(B)$ is different from
the $K$-algebra action on $D_n(B)$ given by scalar multipication. For 
every $n$, there is a natural projection map of $K$-algebras $T:D_{n+1}(B)
\map D_n(B)$ given by 
$T(b_0 + \cdots + b_{n+1}\epsilon^{n+1}) = b_0 + \cdots +b_n \epsilon^n.$
}

For any $K$-algebra $B$ the $n$-th jet space functor of $X$, denoted 
$J^nX$, is defined to be  
$$J^nX(B)=X(D_n(B)).$$

Then the functor $J^nX$ is representable by a scheme over $\Spec K$
\cite{Bu1}(Chapter 3, Proposition $3.16$). The 
$K$-algebra morphism $T$ induces a canonical map of functors $u:J^{n+1}X \map
J^nX$ for all $n$. Also, there exists a canonical derivation on sheaves
$\partial : \Ou_{J^{n-1}X} \map u_* \Ou_{J^nX}$.
A scheme $X$ over $\Spec K$ is called a {\it $D$-scheme} if there exists a 
derivation $\partial : \Ou_X \map \Ou_X$ on the structure sheaf of $X$.
By a {\it $D$-group scheme $G$} over $\Spec K$, we will understand the scheme $G$ 
endowed with a $D$-scheme structure that is also compatible with the group 
scheme structure of $G$.

Let $G$ be a smooth commutative group scheme of relative dimension $g$ over 
$\Spec K$, then $J^nG$ has relative dimension $g(n+1)$ over $\Spec K$. 
Let $\bX_n(G)$ be the set of group scheme morphisms from $J^nG$ to $\bb{G}_a$.
Such additive characters were used by Manin to prove the Lang-Mordell conjecture
for abelian varieties over function fields \cite{Manin}. Later Buium gave 
a different proof of the result, also using Manin characters \cite{B4}.

Since $\bb{G}_a$ has a $K$-linear structure on it, $\bX_n(G)$ naturally is a 
$K$-vector space. For every $n$, the pull-back of the morphism $u:J^{n+1}G
\map J^nG$, denoted by $u^*$, induces a $K$-linear map of vector spaces 
$u^*: \bX_n (G) \map \bX_{n+1}(G)$. Set 
$$\bX_\infty(G):= \varinjlim_{u^*} \bX_n(G).$$
\cblu{The derivation $\partial$ naturally makes $\bX_\infty(G)$ into a 
$K\{\partial\}$-module (for more details see Section $2.2$) .
}


We call a differential character $\Theta$ of order $n$ to be {\it primitive}
 if $\Theta\notin u^*\bX_{n-1}(G)+\partial\bX_{n-1}(G)$. 
The definition of such primitive characters means that $\Theta$ is not 
obtained by any combination of pullback by $u$ and applying derivatives to any
lower order characters. In Corollary \ref{xprim}, we will show that the 
space of such primitive characters (definition as in Section \ref{primchar}) is 
of dimension $g$. Also we consider a quantity $\mup$, called the {\it upper 
splitting number} which is defined in Section \ref{primchar}.

In \cite{Bu1}, Chapter 5, Theorem $1.1$, Buium shows that in the case when $A$ 
is an abelian variety over 
$\Spec K$ (hence $\dim_K \Ext (A, \bb{G}_a) =g$), $\bX_\infty(A)$ is
 generated by $\bX_{g+1}(A)$ as a $K\{\partial\}$-module. He shows the above by 
 using the theory of universal extensions of abelian varieties.
The following is our first
main result in the general case of a  smooth connected commutative group scheme 
$G$ over $\Spec K$. We remark that we use purely differential algebraic 
methods to prove our statement.

\begin{theorem}
    \label{mt}
    Let $G$ be a smooth connected commutative group scheme of relative dimension $g$ over $\Spec K$ and $r= \dim_K \Ext(G,\Ga)$. Then 
$\bX_{\infty}(G)$ is generated by $g$ primitive 
characters of order at most $r+1$ as a $K\{\partial\}$-module.  
In other words, $m_u \leq r+1$.  \\
\end{theorem}

In Section $2.1$ of \cite{B4}, Buium shows that the kernel group scheme 
$K(A)$ of differential characters is a finite dimensional $D$-group scheme 
when $A$ is an abelian scheme over $\Spec K$. His construction again uses the 
theory of universal extensions of abelian schemes.
Our next main result, using purely differential algebraic methods,
 is the generalization of the above result for a general
$G$ as follows:  

\begin{theorem}
\label{K(G)}
Let $G$ be as in Theorem \ref{mt}.
There exists a canonical finite dimensional $D$-group scheme $K(G)$ (constructed in Section \ref{Kernel-Scheme}) 
which is a vectorial extension of $G$. In particular $K(G)$ satisfies the
following short exact sequence of group schemes over $\Spec K$ 
$$
0 \longrightarrow L(G) \stk{\iota}{\longrightarrow} K(G) 
\longrightarrow G \longrightarrow 0
$$
where $L(G)$ is a vector group of relative dimension $(\mup-1)g$ over $S$, with 
$\mup$ denoting the upper splitting number of $G$ (defined in Section $3$).
\end{theorem}

\color{black}






\section{Jet spaces}
Here we recall some of the basic notions in differential algebra. For a more
detailed discussion, we refer to \cite{Bu1}.
\cy{Let $K$ be a field with a derivation $\partial$ on it and $S=\Spec K$.}
Consider $X$ and $Y$ be schemes over $S$. We say a pair $(u,\partial)$ is a 
{\it \cy{kernel}}, and write
$Y \stk{(u,\partial)}{\longrightarrow} X$, if $u: Y \map X$ is a map of 
schemes over $S$ and $\partial: \Ou_X \map u_*\Ou_Y$ is a
derivation of sheaves lifting the derivation on the field $K$.
\cblu{The above is equivalent to giving a morphism $Y \map J^1X$ of 
$S$-schemes, where $J^1X$ is the first jet space defined later in this section.}

\cy{By a {\it prolongation over $S$, } we mean two kernels of the form $$X_1\xrightarrow{(u_1,\partial_1)} X_0\xrightarrow{(u_0,\partial_0)} X_{-1}, $$ such that the following diagram commutes
 $$
 \xymatrix{
 \mathcal{O}_{X_{-1}}\ar[r]^{\partial_0}\ar[d]^{u_0} & {u_0}_{*}\Ou_{X_0}\ar[d]^{u_1}\\
 {u_0}_{*}\Ou_{X_0}\ar[r]^{\partial_1} & {u_0}_{*}{u_1}_{*}\Ou_{X_1},
 }
  $$ 
where $X_i$s are schemes over $S$ for $i=-1,0,1$.}
 A {\it prolongation sequence over $S$} is a sequence of prolongations
        $$
        \xymatrix{
        S & T^0 \ar_-{(u,\partial)}[l] & T^1 \ar_-{(u,\partial)}[l] &  T^2 \ar_-{(u,\partial)}[l] & 
\cdots\ar_-{(u,\partial)}[l] },
        $$
where each $T^n$ is a scheme over $S$. We will often use the notation $T^*$ or $\{T^n\}_{n \geq 0}$.

\cy{Let $S^*$ denote the prolongation sequence 
$$
        \xymatrix{
        S & S \ar_-{(id,\partial)}[l] & S \ar_-{(id,\partial)}[l] & 
\cdots\ar_-{(id,\partial)}[l]},
        $$}
where the derivation $\partial$ on $\Ou_S$ is induced from the fixed derivation
on $K$.
Prolongation sequences over $S^*$ form a category $\mcal{C}_{S^*}$, where a morphism $f: T^*\to U^*$ is
a family of morphisms $f^n: T^n\to U^n$ commuting with both the $u$ and 
$\partial$, in the evident sense.

\cblu{
For any $K$-algebra $B$, let  $D_n(B):=\frac{B[\epsilon]}{(\epsilon^{n+1})}$ be the ring of truncated polynomials of length $n+1$ with coefficients in $B$. 
Consider the canonical ring homomorphism, called the 
{\it Hasse-Schmidt derivation} $\exp_\partial: K \map D_n(K)$, which is
induced by $\partial$ and is given by
\begin{align}
\nonumber
\exp_\partial(r)=  r+ \frac{\partial r}{1!} \epsilon + \cdots + 
\frac{\partial^n r}{n!} \epsilon^n,
\end{align}
for all $r \in K$. For any $K$-algebra $B$, $\exp_\partial$ canonically
makes $D_n(B)$ into a $K$-algebra. Note that when $\partial$ is a non-trivial
operator on $K$, then the $K$-algebra structure on $D_n(B)$ is different from
the $K$-algebra action on $D_n(B)$ given by scalar multipication.
For every $n$, consider $T: D_{n+1}(B) \map D_n(B)$ to be the $K$-algebra map
given by 
$$T(b_0 + \cdots + b_{n+1}\epsilon^{n+1}) = b_0 + \cdots +b_n \epsilon^n.$$
}

We define \cy{the $n$-th jet space functor of an $S$-scheme $X$ as
        $$
        J^nX (B) := \Hom_S(\Spec D_n(B),X),
        $$}
for any $K$-algebra $B$.
Then $J^nX$ is representable by an $S$-scheme \cite{Bu1}(Chapter 3, Proposition
$3.16$) and $J^*X:= \{J^nX \}_{n \geq 0}$ forms a prolongation sequence which is
called the {\it canonical prolongation sequence}. 

Here we will give an explicit description of $J^nX$. For details, the reader
may see Lemma $3.4$ in \cite{Bu1}.
Suppose $X = \Spec K[\bx]$ be an affine space where 
$\bx = \{x_i\}_{i \in \mcal{I}}$ is a collection of indeterminates indexed by 
$\mcal{I}$. 
Then $J^nX$ is given
by $J^nX \simeq \Spec K[\bx,\Del(\bx),\dots, \Del^n(\bx)]$ where for each $j$, 
$\Del^j(\bx)$ denotes the collection of new indeterminates 
$\{\Del^j(x_i)\}_{i \in \mcal{I}}$.
For each $n$, the morphism $u:J^nX \map J^{n-1}X$ is induced from
the natural inclusion map of $K$-algebras  
$$K[\bx,\Del(\bx),\dots, \Del^{n-1}(\bx)] \inj  K[\bx, \Del(\bx),\dots,\Del^n(\bx)].$$
The derivation $\partial : K[\bx,\Del(\bx),\dots,\Del^{n-1}(\bx)]
\map K[\bx,\Del(\bx),\dots, \Del^n(\bx)]$ is induced from setting 
$\partial(\Del^j( \bx)) = \Del^{j+1}(\bx)$  for all $j=0,\dots, n-1$.

Now suppose $X = \Spec A$ be an affine scheme over $S$ given by $A = 
K[\bx]/I$ where $I \subset K[\bx]$ an ideal.
Then $J^nX \simeq
\Spec J_nA$ where 
\begin{align}
J_n A = K[\bx,\Del(\bx),\dots ,\Del^n( \bx)]/(I, \partial I, \partial^2 I ,\dots
\partial^n I) 
\end{align}
and for any subset $L \subset K[\bx,\dots \Del^{n-1}(\bx)]$ we set
$$
\partial L := \{\partial f ~|~ f \in L\} \subset K[\bx,\dots, \Del^n(\bx)].
$$
Then the canonical map of $K$-algebras $J_{n-1}A \map J_nA$ induces the 
projection map of $S$-schemes $u:J^nX \map J^{n-1}X$.
On the other hand the natural derivation $\partial:J_{n-1}A \map J_nA$ 
induced from $\partial(\Del^i( \bx) = \Del^{i+1}(\bx)$ for all $i=0,\dots n-1$ gives
the derivation of sheaves $\partial:\Ou_{J^{n-1}X} \map u_*\Ou_{J^nX}$. 

When $X$ is a general scheme and suppose $\{U_i\}_{i\in I}$ is an open 
affine cover of $X$, then $\{J^nU_i\}_{i\in I}$ naturally forms an open 
cover for $J^nX$. The projection map $u$ canonically glues over the
affine open cover $\{J^nU_i\}_{i \in I}$ of $J^nX$ and so does the derivation
$\partial$. In this manner, the system of schemes $J^*X$ forms a prolongation 
sequence.

Also note that if $G$ is 
a group scheme, from the above functor of points definition,
 $J^nG$ is also naturally a group scheme for all $n$.
Then the  canonical prolongation sequence satisfies the following
universal property \cy{\cite{Bu1} (Page 70)}: for any $T^* \in \mcal{C}_{S^*}$ and $X$ a scheme over
$S^0$, we have
        $$
        \Hom_S(T^0,X) = \Hom_{\mcal{C}_{S^*}}(T^*, J^*X).
        $$

\subsection{Character groups}
Given a prolongation sequence $T^*$ we define its shift $T^{*+n}$ by
$(T^{*+n})^j:= T^{n+j}$ for all $j$, that is
        $$
        \Spec K \stk{(u,\partial)}{\longleftarrow} T^n 
\stk{(u,\partial)}{\longleftarrow} T^{n+1}\cdots. 
        $$
We define a {\it $\partial$-morphism of order $n$} from $X$ to $Y$ to be a
morphism $J^{*+n}X \map J^*Y$ of prolongation sequences. \cy{ Given a group scheme $G$ over $S$,
we define a {\it differential character $\Theta$ of order $n$}, 
to be a $\partial$-morphism of group schemes of order $n$ from $G$ to $\Ga$
where the map between the prolongations sequences are group scheme morphisms.
Then by the universal property satisfied by a canonical prolongation sequence, an order $n$ character is equivalent to a morphism
$\Theta: J^nG \map \Ga$ of group schemes over $S$. Set
        $$
        \bX_n(G)=\Hom(J^nG,\Ga)
        $$
to be the group of differential characters of order $n$ of $G$.
Note that $\bX_n(G)$ comes with a natural
$K$-vector space structure since $\Ga$ has a $K$-linear structure on it.} Also the inverse system
$J^{n+1}G \stk{u}{\map} J^nG$ defines a directed system
        $$
       \cdots \stk{u^*}{\map} \bX_n(G) \stk{u^*}{\map} \bX_{n+1}(G) \stk{u^*}{\map}\cdots,
        $$
via pull back by $u^*$. Note that $u^*$ is injective.
We will say $\Theta$ is of {\it strict order} $n$ if $\Theta \in
\bX_n(G) \backslash u^*\bX_{n-1}(G)$.
We define $\bX_\infty(G):=\varinjlim \bX_n(G)$ to be the module of 
differential characters of $G$.

\subsection{The comultiplication map} 
\label{comult}

For any morphism of $S$-schemes $f:X \map Y$, let $f^*:\Ou_Y \map f_*\Ou_X$
denote the corresponding map of sheaves. Here for brevity, we will write
the map of sheaves as $f^*:\Ou_Y \map \Ou_X$ which is an abuse of notation but 
will help us to keep the use of symbols down in this section.

For any group scheme, let $m^*_G: \Ou_G \map \Ou_G \otimes \Ou_G$ denote the 
comultiplication map of sheaves associated to the group structure of $G$.
For a kernel $(u,\partial):H \map L$ of group schemes, we will say 
$\partial$ is compatible with the group structures, if the following diagram
of sheaves is commutative
$$\xymatrix{
\Ou_H \ar[r]^-{m^*_H} & \Ou_H \otimes \Ou_H \\
\Ou_L \ar[r]^-{m^*_L} \ar[u]^\partial & \Ou_L \otimes \Ou_L \ar[u]_{\partial 
\otimes u^* + u^* \otimes \partial}.
}$$
Hence for any $y \in \Ou_L$, we have
\begin{align}
\label{kk}
m^*_H(\partial y) = (\partial \otimes u^* + u^* \otimes \partial)(m^*_L(y)).
\end{align}

Given an $S$-group scheme $G$, for each $n$  we have the following short exact 
sequence of $S$-schemes
\begin{align}
0\map N^nG \stk{\iota}{\map} J^nG \stk{u}{\map} G \map 0,
\end{align}
where $N^nG$ is the kernel of the projection map $u:J^nG\rightarrow G$. For each $n$, consider $m^*_{N^nG}: \Ou_{N^nG} \map \Ou_{N^nG} \otimes \Ou_{N^nG}$ 
the comultiplication map induced from subgroup scheme structure of $N^nG$
inside $J^nG$.

Let $\bX_0(T)= \Hom(T, \Ga)$ denote the $K$-vector space of additive characters
of any group scheme $T$.
The following is a standard result.
\begin{lemma}
\label{compo}
Let $(u,\partial):H \map L$ be a prolongation of $S$-group schemes, such that
$\partial$ is compatible with the group structures. Then $\partial \bX_0(L) 
\subseteq \bX_0(H)$.
\end{lemma}
\begin{proof}
Let $\Ga = \Spec K[x]$. Then any $f \in \bX_0(L)$ satisfies the following
commutative diagram of sheaves
$$\xymatrix{
\Ou_L \otimes \Ou_L & & \Ou_{\Ga} \otimes \Ou_{\Ga}
 \ar[ll]_{ (f^* \otimes f^*)} \\
\Ou_L \ar[u]^{m^*_L} & & \Ou_{\Ga}, \ar[ll]^{f^*} \ar[u]_{m^*_{\Ga}}
}$$
where if $x$ is a global coordinate function of $\Ga$, $m^*_{\Ga}$ on 
$\Ou_{\Ga}$ is induced by $m_{\Ga}^*(x)= x 
\otimes 1 + 1 \otimes x$. Then from the above diagram, we get 
$$m^*_L(f^*(x)) = f^*(x) \otimes 1 + 1 \otimes f^*(x).$$

By applying (\ref{kk}), it is easy to see that $\partial f^*$
satisfies
$$ m^*_H(\partial f^* (x)) = (\partial f^* \otimes \partial f^*)(m^*_{\Ga}(x)),$$
which shows that $\partial f \in \bX_0(H)$.
\end{proof}

\cblu{
Since $\bb{G}_a$ has a $K$-linear structure on it, $\bX_n(G)$ naturally is a 
$K$-vector space. 
For any $\Theta \in \bX_{n}(G)$, by the above discussion 
$\partial \Theta$ is a differential 
character of order $n+1$, that is $\partial \Theta \in \bX_{n+1}(G)$.
Let $K\{\partial \}$ denote the non-commutative ring of operators 
induced by the relation
$\partial . r = r. \partial + r'$ where $r'$ is the derivative of $r \in K$ with
respect to the fixed derivation on $K$.
Then by the above discussion $\bX_\infty(G)$ naturally 
is a $K\{\partial\}$-module.
}

\subsection{Representation of differential characters in local coordinates}

In this subsection, we will develop the necessary theory on representing 
differential characters $\Theta$ in terms of local coordinate functions. The
following discussions and results will be  used in  section \ref{primchar}.


Consider the affine scheme $U = \Spec C$ over $S$ where $C=\frac{K[\bx]}{(\bff)}$ and $\bx= x_1,\dots, x_l$ 
are a system of indeterminates and the ideal $(\bff) = (f_1, \dots , f_q)$, 
$f_i \in K[\bx]$ for all $i$.
Then $J^nU = \Spec J_nC$ where 
$J_nC = \frac{K[\bx,\Del(\bx),\cdots, \Del^{n}(\bx)]}{(\bff, \partial \bff, \dots , 
\partial^n \bff)}$ and $(\partial^i\bff)= (\partial^i f_1,\dots, 
\partial^i f_q)$.
Also let $e:S \map U$ be a section such that the
associated map of $K$-algebras satisfies $\iota^*: C \map K$ such that 
$\iota^*(\bx)
=0$ (the notation here means $\iota^*(x_j) = 0$ for all $j=1,\dots l$).

Define the affine closed subscheme $N^nU = J^nU \times_{U,e} S$. Then we 
have $N^nU \simeq \Spec N_nC$ 
where $N_nC = \frac{K[\Del(\bx),\cdots, \Del^n(\bx)]}{(\partial \bff, \dots , 
\partial^n \bff)}$
and the canonical morphism $\iota: N^nU \map J^nU$ is associated to the 
$K$-algebra map 
$\iota^*: J_nC \map N_nC$ given by $\iota^*(\bx)=0, \iota^*(\Del^j(\bx)) = 
\Del^j(\bx)$ for all 
$j =1, \dots , n$. Indeed $N^nU$ is the closed subscheme associated to the
ideal $(\bx) \subset \Ou(J^nU)$.

In the case when $U$ is an affine open subscheme of a group scheme
$G$ such that $U$ contains the identity section $e:S \map G$, we have
$N^nU \simeq N^nG$ for all $n$.

\begin{lemma}
\label{deltaT}
Let $f \in \Ou(J^nU)$ such that $\iota^*f =0$, that is $f \in (\bx) \subset 
\Ou(J^nU)$. Then for all $s$ we have 
$$
\partial^s(\bx.J_nC) \subset \bx . J_{n+s}C + u^*J_{n+s-1}C. \\
$$

In particular
$$
\iota^* (\partial^s(\bx . J_nC)) \subset u^* J_{n+s-1}C \subset J_{n+s}C.
$$
\end{lemma}

\begin{proof}
It is sufficient to prove the result for the polynomial algebra $C=K[\bx]$. 
For any $f \in \Ou(J^nU)$ such that $\iota^*f =
0$, we have $f \in (\bx)$. Hence there exists elements $h_1, \dots  h_l \in 
\Ou(J^nU)$ such that 
$$
f = h_1 x_1 + \cdots h_l x_l.
$$
We will now apply induction on $s$. Clearly the result holds for 
$s=0$. Assume it is true for $s$. Then by the induction hypothesis we have 
$$
\partial^s f = f_{s_1}+ f_{s_2},
$$
where $f_{s_1} \in (\bx)$ and $f_{s_2} \in \Ou(J^{n+s-1}u)$. Since $f_{s_1}
 \in (\bx)$, there
exists $t_1,\dots, t_l \in \Ou(J^{n+s-1}U)$ such that $f_{s_1}= t_1 x_1 + \cdots
t_l x_l$. Applying the derivation $\partial$ to the above we obtain
$$\partial^{s+1} f = f_{{s+1}_1} + f_{{s+1}_2}$$
where $f_{{s+1}_1} = (\partial t_1) x_1 + \cdots (\partial t_l) x_l$ and  
 $f_{{s+1}_2} = t_1 \Del(x_1) + \cdots t_1 \Del(x_l) + \partial (f_{s_2})$.
Then note that $f_{{s+1}_2} \in \Ou(J^{n+s}U)$ and $f_{{s+1}_1} \in (\bx)$
and this proves our result.
\end{proof}


Let $G$ be a smooth group scheme of dimension $g$ over $S$.
Then there is an open affine subscheme $U$ of $G$ containing the identity section $e$, such that, we have an \'etale morphism of \cy{affine} schemes $U\rightarrow \bb{A}^g$, where $\Ha=\Spec(K[\bx]),$ $\bx=(x_1,\dots,x_g)$. 
Then $\bx$ is referred to as a system of \'etale coordinates around the 
identity section of $G$. By \cite{Bu1} we have the
following isomorphism of schemes
$$J^nU\cong U \times_{\Ha} J^n{\Ha}$$
for all $n$. For any $K$-algebra $B$, we have $$J^n\Ha(B)=\Hom_K(K[\bx],D_n(B)).$$
Hence we obtain $J^n\Ha\cong \Spec K[\bx,\Del(\bx),\cdots,\Del^n(\bx)]$, where $\partial^i\bx=(\Del^i( x_1),\cdots, \Del^i(x_n)).$ Therefore, we have 
$$J^nU \simeq U \times_{\Ha} J^n{\Ha} \simeq U \times_{\Ha}\Spec(K[\bx,\cdots,\Del^n(\bx)]).$$
Hence $\{\Del^i(\bx)\}$ form an \'etale coordinate system on $J^nG$ around the identity section. 
Now $N^nG$ can be obtained as the following fibre product
$$\xymatrix{
N^nG=J^nU\times_US\ar[r]^-{}\ar[d] & J^nU \ar[d]\\
 S \ar[r]^-{e} & U ,
}$$
where the product is taken over the identity section $S\xrightarrow{e} U$. 
By composition, we have a section $S\xrightarrow{e}U\rightarrow \Ha$ of $\Ha$, 
which is given by the map of $K$-algebras $K[\bx]\rightarrow K$ as $\bx\mapsto0$.  
Hence we have 
$$N^nG =J^nU\times_{U,e} S=(J^n\Ha\times_{\Ha} U)\times_U S=J^n\Ha\times_{\Ha} S\cong \Spec(K[\Del(\bx),\cdots,\Del^n(\bx)]).$$
Hence the functions $\{\Del(\bx),\cdots , \Del^n(\bx)\}$ form a coordinate system
for the vector group $N^nG$.

For a connected smooth commutative group scheme $G$, by Theorem $2$ in 
\cite{Brion}, $G$ satisfies the short exact sequence of group schemes
$$
0 \longrightarrow H \stk{\kappa}{\longrightarrow} G \longrightarrow A 
\longrightarrow 0, 
$$
where $H$ is a connected smooth affine algebraic group and $A$ is an abelian 
scheme. Since $\bX_0(A)=\{0\}$, we have $\bX_0(G) \stk{\kappa^*}{\inj}
 \bX_0(H)$. Now by Theorem 
$5.3.1$ of \cite{Brion}, $H$ is a \cblu{split} extension of a 
\cblu{vector group}
 $V$ by a commutative group scheme of multiplicative type $M$
$$
0 \map M \map H \map V \map 0.
$$
Since the characteristic of the base field $K$ is $0$, by \cite{Serre} Page 
$171$, $V \simeq \Ga^m$ for some $m$. 

Also note that for any multiplicative
type group scheme  $M$, there
are no non-trivial group scheme morphisms from $M$ to $\Ga$, that is $\bX_0(M)=
\{0\}$. Hence we have $\bX_0(H) \simeq \bX_0(V)$. Therefore combining
with the above we have $\bX_0(G) \stk{\kappa^*}{\inj} \bX_0(V)\simeq
\mathrm{Mat}_{1\times g}(K)$.



Let $\bx$ be 
an \'{e}tale coordinate coordinate system  of $G$ around the identity 
section. Hence any additive character $\Psi \in \bX_0(G)$
can be represented as $\Psi = A. \bx= a_1 x_1 + \cdots + a_g x_g$ where 
$A = (a_1 \cdots a_g) \in 
\mathrm{Mat}_{1\times g}(K)$.

We recall the following result.
\begin{proposition}
\label{Nn}
\cy{Let $G$ be a smooth commutative group scheme over $S$. Then}
for all $n$ we have
\begin{enumerate}
\item $N^nG \simeq \Ga^{ng}$.

\item The $K$-vector space $\bX_0(N^nG)$ is of dimension $ng$.
\end{enumerate}
\end{proposition}

\begin{proof}
We refer to Chapter $3$ in \cite{Bu1} for proof.
\end{proof}


Recall the exact sequence of $K$-vector spaces
\begin{align}
\label{io}
0 \longrightarrow \bX_0(G) \longrightarrow \bX_n(G) 
\stk{\iota^*}{\longrightarrow} \bX_0(N^nG)
\end{align}
which induces the map $\iota^*: \bX_n(G)/\bX_0(G) \map \bX_0(N^nG)$. 

Note that $\iota^*\Theta$ is a $K$-linear morphism from
$N^nG~ (\simeq \Ga^{ng})$ to $\Ga$. Therefore  
$$\iota^*\Theta=A_n.\Del^n(\bx)+A_{n-1}.\Del^{n-1}(\bx)+\cdots +A_1.\Del(\bx),$$ 
for some $A_i \in \mathrm{Mat}_{1\times g}(K)$,  $i=1,\dots, n$. 

On the other hand, any element $\Psi \in \bX_0(G)$ is of the form
$A_0. \bx$ for some $A_0 \in \mathrm{Mat}_{1\times g}(K)$.


For a given differential character $\Theta \in \bX_n(G)$, we wish to now 
find its expression
in terms of the \'{e}tale coordinates $\bx,\Del(\bx),\cdots, \Del^n(\bx)$ of 
$J^nU$ where $U$ is an open subscheme of $G$ containing the identity section
that admits an \'etale map $U \map \Ha = \Spec R[\bx]$.
Note that $\iota^*\Theta$ is obtained from $\Theta(\bx,\dots, \Del^n(\bx))$ 
by setting $\bx=(0,\dots, 0)$. Therefore combining the above discussion 
with (\ref{io}) we have 
\begin{align}
\label{fullexp}
\Theta (\bx,\Del(\bx),\dots \Del^n(\bx))=A_n.\Del^n(\bx)+A_{n-1}.\Del^{n-1}(\bx)+\cdots +A_1.\Del(\bx)+A_0 . \bx + 
f,
\end{align}
where $f \in (\bx) \subset \Ou(J^n(U))$. 

Given a differential character $\Theta \in \bX_n(G) \backslash \bX_{n-1}(G)$,
we define the {\it leading linear term} of $\Theta$ with respect to the 
chosen \'{e}tale coordinate to be 
$$
L(\Theta) := A_n.\Del^n(\bx)
$$
and $n$ will be called the {\it order} of $L(\Theta)$.

\section{Primitive differential characters of $G$.}
\label{primchar}
The aim of this section is to show that $\bX_\infty(G)$ is a free 
$K\{\partial\}$-module of rank $g$ and is generated by $\bX_{r+1}(G)$ where
$r = \dim_K \Ext^1(G, \hG)$.
Given a smooth group scheme $G$ over $S$,  consider 
the following short exact sequence of \cy{smooth group schemes over $S$ for 
all $n$
\begin{align}
0 \map N^nG \stk{\iota}{\longrightarrow} J^nG \stk{u}{\longrightarrow} G \map 0,
\end{align}
where $N^nG$ is the kernel of the projection map $u:J^nG\rightarrow G$.}
Applying $\Hom(-,\Ga)$ to the above short exact sequence gives us the following exact sequence of $K$-vector spaces
\begin{align}
\label{les}
0 \map \bX_0(G) \map \bX_n(G)  \stk{\iota^*}{\longrightarrow}  \bX(N^nG) 
\stk{\Delta}{\longrightarrow} \Ext(G,\Ga),
\end{align}
where $\bX_0(G)=\Hom(G,\Ga)$ for any group scheme $G$. Let $\dim_K \Ext(G,\Ga) = r$. Also let us define $\bI_n := \mb{image}(\Delta)$. We 
define $\mlow$ as the {\it lower splitting number} if $\bX_{\mlow}(G) \ne 
\{0\}$ and 
$\bX_{\mlow-1}(G) = \{0\}$. 

We define the {\it upper splitting number} as the 
smallest number $\mup$ such that $\dim_K \bI_n$ is constant for all $n \geq 
\mup-1$. Note that $\mup$ exists since for all $n$, $\dim_K \bI_n \leq 
\dim_K \Ext(G,\Ga) < \infty$. We will see in Proposition \ref{mlmu} that 
this definition is equivalent to the least $m=\mup$ such that 
$\bX_{m}(G)$ generates $\bX_\infty(G)$ as a $K\{\partial\}$-module.


We define a differential character $\Theta \in \bX_n(G)$ to be 
{\it primitive} if 
$$\Theta \not \in u^* \bX_{n-1}(G) + \partial\bX_{n-1}(G).$$
\cy{Hence} the image of $\Theta$ in the quotient 
$\bX_n(G)/(u^*\bX_{n-1}(G) + \partial\bX_{n-1}(G))$ is non-zero. 
We recall the exact sequence (\ref{les})
$$0 \map \bX_0(G) \map \bX_n(G) \stk{\iota^*}{\longrightarrow} \bX(N^nG) 
\stk{\Delta}{\longrightarrow} 
\Ext(G,\Ga).$$
Set $h_0=0$ and
define $h_i= \dim \bI_i - \dim \bI_{i-1}$ for all $i \geq 1$. Then note 
that $h_1$ is the $K$-dimension of $\bI_1$. For each $i$, 
we define the primitive dimension of $\bX_i(G)$ to be 
$l_i:= \dim \bX_i(G) - \dim (u^*\bX_{i-1}(G) + \partial\bX_{i-1}(G))$.

We say the set 
$\mathbbm{B}_i:=\{\Theta_{if}\}^{\mult(i)}_{f=1}$, where $\mult(i):=|\mbB_i|$, 
is a {\it primitive subset} 
for $\bX_i(G)$ if their images in $\bX_i(G)/(u^*\bX_{i-1}(G) + 
\partial\bX_{i-1}(G))$ are distinct and form a $K$-basis of the quotient $\bX_i(G)/(u^*\bX_{i-1}(G)_K + 
\partial\bX_{i-1}(G))$. 

Let $\mathbbm{B}$
be the set such that it generates $\bX_{\infty}(G)$ as a 
$K\{\partial\}$-module. 
For each $i$, denote $\mathbbm{B}_i = \mathbbm{B} \cap (\bX_i(G) 
\backslash u^*\bX_{i-1}(G))$. We define $\mathbbm{B}$ to {\it 
primitively generate} $\bX_\infty(G)$ if, for all $i$, $\mathbbm{B}_i$ is
a primitive subset for $\bX_i(G )$. Note that one can always construct such a
$\mathbbm{B}$ by taking the union of $\mathbbm{B}_i$ over all $i$, where $\mathbbm{B}_i$ is 
a primitive subset for $\bX_i(G)$ and $\mult(i)=l_i$.

For all $n \geq i$ set $S_n(\mathbbm{B}_i) := \{\partial^{h}\Theta \mid  
\mb{ for all } 0 \leq h 
\leq (n-i) \mb{ and } \Theta \in \mathbbm{B}_i\}$. \cblu{Define 
$$\bXp(G) := \varinjlim \bX_n(G)/\langle \partial\bX_{n-1}(G)\rangle,$$
where $\langle \partial \bX_{n-1}(G)\rangle$ is the $K$-vector space generated
by elements $\partial \Theta$ for all $\Theta \in \bX_{n-1}(G)$.}

\begin{lemma}
\label{partlead}
Let $\Theta \in \bX_n(G) \backslash u^*\bX_{n-1}(G)$ such that $L(\Theta) =
A_n. \Del^n(\bx)$ for some non-zero $A_n \in \mathrm{Mat}_{1\times g}(K)$. 
Then $L(\partial^s(\Theta)) = A_n. \Del^{n+s}(\bx)$.
\end{lemma}
\begin{proof}
We will prove using induction on $s$. The result is true for $s=0$ and hence
under the induction hypothesis we have $L(\partial^s(\Theta)) = A_n.
\Del^{n+s}(\bx)$. Then by (\ref{fullexp}) we have 
$$
\partial^s(\Theta) = A_n \Del^{n+s}(\bx)+ B_{n+s-1}.\Del^{n+s-1}(\bx) + \cdots +
B_0 \bx + f 
$$
for some $B_i \in \mathrm{Mat}_{1\times g}(K)$ for all $i$ and $f \in (\bx)$.
Applying our derivation $\partial$ to the above equation we obtain
\beqar
\partial^{s+1}(\Theta) &=& (A_n.\Del^{n+s+1}(\bx)+ 
\partial A_n .\Del^{n+s}( \bx)) + (B_{n+s-1}.\Del^{n+s} (\bx) \\
& & + \partial B_{n+s-1}.\Del^{n+s-1}(\bx)) + \cdots  
 + (B_0.\Del( \bx) + \partial B_0. \bx) + \partial f.
\eeqar
Collecting coefficients we get
\beqar
\partial^{s+1}(\Theta) = A_n . \Del^{n+s+1}(\bx) + (\partial A_n + B_{n+s-1}).
\Del^{n+s}(\bx )+ \cdots + \partial B_0.\bx + \partial f.
\eeqar
By Lemma \ref{deltaT}, $\partial f = f_1 + f_2$ where $f_1 \in (\bx)$
and $f_2 \in \Ou(J^{n+s-1}U)$.
Since $s \geq 1$, note that $\ord{L(\Theta)} \geq 1$ which implies that 
$L(\Theta) \notin (\bx)$. Here $A_n.\Del^{n+s+1}(\bx) \in \Ou(J^{n+s+1}U) \backslash
\Ou(J^{n+s}U)$ where as $f_2 \in \Ou(J^{n+s-1}U)$. Therefore the leading
coefficient of $\partial^{s+1}(\Theta)$ must be $A_n.\Del^{n+s+1}(\bx)$.
\end{proof}

\begin{proposition}
\label{first}
If $\mathbbm{B}_i$ is a primitive subset for $\bX_i(G)$, then the set $S_n(\mathbbm{B}_{i})$  is $K$-linearly independent inside $\bX_n(G)$ for all $n\geq i$.
\end{proposition}

\begin{proof}
It is sufficient to consider $n \geq 1$. Let 
$\mbB_i=\{\Theta_{ik}\}_{k=1}^{l_i}$ and denote $\partial^h\mbB_i=\{\partial^h\Theta_{ik}\}_{k=1}^{l_i}$. Then
$S_n(\mbB_i)=\mbB_i\cup \partial\mbB_i\cup...\cup \partial^{n-i}\mbB_i$.

Let $L(\Theta_{ik}) = A_{ik}.\bx^{(i)}$ for all $k=1, \dots , l_i$.
Since $\mbB_i=\{\Theta_{ik}\}_{k=1}^{l_i}$ is a $K$-basis for $\bX_i(G)$, 
we must have that $\{A_{ik}\}_{k=1}^{l_i}$ are $K$-linearly independent.
Also for all $n \geq i$, by Lemma \ref{partlead} we must 
have
$L(\partial^{n-i} (\Theta_{ik})) = A_{ik}.\bx^{(n)}.$

Suppose $n$ be the least positive integer for which $S_n(\mbB_i)$ is not $K$-linearly independent. Then we must have $\gamma_1,\dots, \gamma_{l_i} \in K$ of which at least 
one of them is non-zero such that
$$\gamma_1L(\partial^{n-i}\Theta_{i1})+\cdots+
\gamma_{l_i}L(\partial^{n-i}\Theta_{il})=0.$$

This implies that we  have
 $$(\gamma_1A_{i1}+\cdots+\gamma_iA_{il}).\Del^n(\bx)=0.$$
But since $A_{i1},\dots, A_{il}$ are $K$-linearly independent, we must have 
$\gamma_1=\cdots=\gamma_l=0,$ which is a contradiction and we are done.     
\end{proof}




\begin{lemma}\label{1}
For any $\Psi\in \bX_m(G)$, if $\partial^i\Psi\in \bX_m(G)$, then $\Psi\in \bX_{m-i}(G)$.
\end{lemma}
\begin{proof}
Let $L(\Psi) = A. \Del^k(\bx)$. Then by Lemma \ref{partlead} we must have 
$L(\partial^i(\Psi)) = A.\Del^{k+i}( \bx)$. Therefore we have $k+i \leq m$
which implies $k \leq m-i$. Hence $\Psi \in \bX_{m-i}(G)$ and we 
are done.
\end{proof}

\begin{proposition}\label{3}
If $\mathbbm{B}_i$ is a primitive subset for $\bX_i(G)$ for all $i$, then $S_n(\mathbbm{B}_{i_1})\cup \cdots \cup S_n(\mathbbm{B}_{i_k})$ is $K$-linearly 
independent for all $\{i_1,...,i_k\}\subset \{0,...,n\}$ over $K$.
\end{proposition}

\begin{proof}
We will prove using induction on $k$. For $k=1$, the result is true by
Proposition \ref{first}. Now assume the result is true for $\{i_1,\dots ,i_{k-1}
\}$.
Without loss of generality, we can assume that $i_1 <i_2<...<i_k$ to be a
strictly increasing sequence. Suppose the result does not hold for $k$. Then 
there exists a smallest $n$ such that 
$S_n(\mathbbm{B}_{i_1})\cup \cdots \cup S_n(\mathbbm{B}_{i_k})$ 
is not 
$K$-linearly independent. Since by the induction hypothesis the set
$S_{i_k}(\mathbbm{B}_{i_1}) \cup \cdots \cup S_{i_k}(\mathbbm{B}_{i_{k-1}})$ is
$K$-linearly independent and $\mathbbm{B}_{i_k}$ is a primitive subset 
for $\bX_{i_k}(G)$, we must have $n > i_k$.

For each $i$, let  $\mbB_i=\{\Theta_{is}\}_{s=1}^{l_i}$. 
Since $n$ is the smallest such integer we must have
atleast one non-zero scalar $a_{hs}$ such that 
\begin{equation}\label{eq:1}\sum_{h=1}^k(\sum_{s=1}^{l_{i_h}}a_{hs}\partial^{n-i_h}\Theta_{i_hs})\in \bX_{n-1}(G).\end{equation}
Define the differential character
 $$\Psi=\sum_{h=1}^k(\sum_{s=1}^{l_{i_h}}a_{hs}\partial^{i_k-i_h}\Theta_{i_hs})\in \bX_{i_k}(G).$$

Applying $\partial$ to the above expression $(n-i_k)$-times we obtain
\begin{align*}
\partial^{n-i_k}\Psi &=\partial^{n-i_k}(\sum_{h=1}^k(\sum_{s=1}^{l_{i_k}}a_{hs}\partial^{i_k-i_h}\Theta_{i_hs}))\\
&=\sum_{h=1}^k(\sum_{s=1}^{l_{i_k}}\partial^{n-i_k}(a_{hs}\partial^{i_k-i_h}\Theta_{i_hs}))\\
&= \sum_{h=1}^k(\sum_{s=1}^{l_{i_k}}a_{hs}\partial^{n-i_h}\Theta_{i_hs}))+
f
\end{align*}
for some $f \in \bX_{n-1}(G)$.
By (\ref{eq:1}) we have the summation to be in $\bX_{n-1}(G)$, hence $\partial^{n-i_k}\Psi \in \bX_{n-1}(G)$. Now by Lemma \ref{1}, we have $\Psi\in \bX_{i_k-1}(G)$ which implies
that $n=i_k$. But this is a contradiction to the fact that $n > i_k$ 
and we are done.
\end{proof}

\begin{proposition}
\label{dim}
$$\dim_K(u^*\bX_{n-1}(G) + \partial\bX_{n-1}(G)) = (n+1)l_0+ nl_1 + \cdots +
2l_{n-1}.$$

\end{proposition}
\begin{proof}
Let $\mathbbm{B}_i$ be a primitive subset for $\bX_i(G)$ for all $i$. Then 
the cardinality of $\mathbbm{B}_i$ is $l_i$ and the set $S_n(\mathbbm{B}_0) 
\cup \cdots \cup S_n(\mathbbm{B}_{n-1})$ spans $u^* \bX_{n-1}(G) + \partial
\bX_{n-1}(G)$ as a $K$-vector space. By Proposition \ref{3}, 
they are $K$-linearly 
independent and hence form a $K$-basis. Then the result follows from the fact
that the cardinality of $S_n(\mathbbm{B}_0) \cup \cdots \cup S_n(
\mathbbm{B}_{n-1})$ is $(n+1)l_0 + nl_1 + \cdots + 2l_{n-1}$.
\end{proof}

\begin{corollary}
\label{dimX}
For all \cy{$n$}, we have $\dim_K \bX_n(G) = (n+1)l_0+ nl_1 + \cdots + 2l_{n-1} + l_n$.
\end{corollary}
\begin{proof}
\cy{We have $\dim_K(\bX_n(G))=\dim_K(u^*\bX_{n-1}(G)+\partial 
\bX_{n-1}(G))+l_n$. Hence, by Proposition \ref{dim}, we obtain our result.}
\end{proof}

\begin{lemma}
\label{l&h}
$(1)$ For $n=0$, we have $l_0 = \dim_K \bX_0(G)$ and $h_0=0$.

$(2)$ For $n=1$, we have $l_1= g-h_1-l_0$.

$(3)$ For all $n \geq 2$, we have $l_n= h_{n-1} - h_n$.
\end{lemma}
\begin{proof}
For each $n$, we have the following exact sequence of $K$-vector spaces
\begin{align}
\label{I_n}
 0 \map \bX_0(G) \map \bX_n(G) \stk{\iota}{\longrightarrow} 
\bX_0(N^nG) \stk{\Delta}{\longrightarrow} \bI_n \map 0.
\end{align}
$(1)$: This is clear since $N^0G$ is a point scheme corresponding to the 
identity section.

\noindent
$(2)$: For $n=1$, by Corollary \ref{dimX}, we have $\dim_K(\bX_1(G))=2l_0+l_1, \dim_K(\bX_0(G))=l_0$, $\dim_K(\bX_0(N^1G))=g$ and $\dim_K \bI_1=h_1-h_0=h_1$, hence from the exact sequence, we get 
\beqar
(2l_0+l_1)-l_0+h_1 &=& g \\
l_0+l_1 &=& g-h_1.
\eeqar

\noindent
$(3)$: We will prove the result by induction.
For $n=2$, from (\ref{I_n}) we have 
\beqar
(3l_0+2l_1+l_2)-l_0+(h_2+h_1) & =& 2g.
\eeqar
Substituting $l_0+l_1 = g-h_1$ we obtain
\beqar
 l_2 &=& h_1-h_2.
\eeqar
Assume the result is true for all $n-1$. We want to show
it for $n$. From Corollary \ref{dimX} and (\ref{I_n}) we obtain,
\beqar
((n+1)l_0+nl_1+\cdots +2l_{n-1}+ l_n)- l_0 + (h_1+ \cdots + h_n) &=& ng.
\eeqar
Now substituting $l_1, l_2, \cdots ,l_{n-1}$ with $h_i$s, from the induction 
hypothesis, we obtain the result.
\end{proof}

\begin{lemma}
For all $n\geq 1$, $h_n$ is a (weakly) decreasing function of $n$.
\end{lemma}
\begin{proof}
We know that $l_n \geq 0$ for all $n$ and hence by Lemma \ref{l&h} the result
follows.
\end{proof}

\begin{corollary}
\label{hn=0}
If $h_N = 0$ for some $N$, then 
\begin{enumerate}
\item $h_n=0$ for all $n\geq N$.

\item $l_n=0$ for all $n\geq N+1$.
\end{enumerate}
\end{corollary}


\begin{proof}
The results  follow immediately from the fact that $l_{n+1}= h_n -h_{n+1}$.
\end{proof}
\begin{lemma}
\label{ee}
For all $n \geq \mup + 1$ we have 
$$\bX_n(G)/u^*\bX_{n-1}(G) \simeq \bX_0(N^n)/u^*\bX_0(N^{n-1}).$$
\end{lemma}
\begin{proof} 
For all $n \geq \mup + 1$, by definition we have $\bI_n=\bI_{n+1}=\cdots$. 
Consider the diagram
$$\xymatrix{
0 \ar[r] & \bX_n(G)/\bX_0(G) \ar[r]^-{\iota^*}& \bX_0(N^nG) \ar[r] & \bI_n 
\ar[r] & 0\\
0 \ar[r] & \bX_{n-1}(G)/\bX_0(G) \ar[u]_{u^*}\ar[r]^-{\iota^*}& \bX_0(N^{n-1}G) 
\ar[u]_{u^*} \ar[r] & \bI_{n-1} \ar[u] \ar[r] & 0.\\
}$$
Then the result follows from snake lemma.
\end{proof}

\begin{proposition}
\label{mlmu}
$(1)$ For all $n \geq m_u+1$, we have $l_n=0$. In other words, there are no 
primitive characters of order greater than $m_u$.

$(2)$ We have $m_l \leq m_u$.
\end{proposition}
\begin{proof}
(1): By the definition of $m_u$, we have $h_{m_u}= h_{m_u+1}= \cdots = 0$. This
implies by Lemma \ref{l&h} that $l_n=0$ for all $n \geq m_u+1$.

(2): By definition, $m_l$ is the least number for which a non-trivial primitive
basis exists, whereas $m_u$ is the largest number for which a primitive subset 
exists. Hence we have $m_l\leq m_u$. 
\end{proof}

Note that as a consequence of Proposition \ref{mlmu},
if $\mathbbm{B}$ primitively generates $\bX_{\infty}(G)$, then $\mathbbm{B}$
can be written as a finite union 
\begin{align}
\mathbbm{B} = \mathbbm{B}_{i_1} \cup \cdots \cup \mathbbm{B}_{i_l},
\end{align}
where $m_l=i_1 < \cdots < i_l = m_u$ and $\mathbbm{B}_{i_j}$ is a primitive
basis for $\bX_{i_j}(G)$ for all $j$.

Now we prove our first main result.

\begin{proof}[Proof of Theorem \ref{mt}]
For any $n$ consider the following sum obtained by applying Lemma \ref{l&h}
$$l_0+ l_1+ \cdots + l_n = g - h_n.$$
By Corollary \ref{hn=0}, for $n$ large enough, $h_n$ is $0$ and hence the sum above is $g$ and
this shows that $\bX_{\infty}( G)$ is generated by $g$ primitive characters.


Note that $\sum h_i \leq r$ and since $h_i$s are a 
weakly decreasing sequence of non-negative integers, we must have $h_i=0$ for
all $i \geq r+1$ and hence $m_u \leq r+1$.
\end{proof}

\begin{corollary}
\label{xprim}
We have
$$\bXp(G)\simeq\bX_{\mup}(G)/\langle\partial \bX_{\mup-1}(G)\rangle .$$
Moreover, $\mathbbm{B}_{i_1} \cup \cdots \cup
\mathbbm{B}_{i_l}$ is a $K$-basis for $\bXp(G)$ and $\dim_K \bXp(G)= g$, where $\mathbbm{B}_{i_j}$ is a primitive subset of $\bX_{i_j}(G)$ and $m_l=i_1<...<i_l=m_u$.
\end{corollary}
\begin{proof}
For all $n \geq \mup$,
$S_n(\mathbbm{B}_{i_1}) \cup \cdots \cup S_n(\mathbbm{B}_{i_l})$
generates $\bX_n(G)$ as a $K$-vector space. Whereas,
$(S_n(\mathbbm{B}_{i_1})\backslash \mbB_{i_1}) \cup \cdots 
\cup (S_n(\mathbbm{B}_{i_l})\backslash \mbB_{i_l})$
generates $\partial \bX_{n-1}(G)$ as a $K$-vector space. Therefore
$\bX_n(G)/ \partial \bX_{n-1 }(G)$ is generated by $\mbB_{i_1} \cup \cdots 
\cup \mbB_{i_l}$ as a $K$-vector space for all $n \geq \mup$ and hence in particular
$\bXp(G) \simeq \bX_{\mup}(G)/\langle\partial \bX_{\mup -1}(G)\rangle$.
The dimension of $\bXp(G)$ is $g$ follows from Theorem \ref{mt}.
\end{proof}

\begin{corollary}
\label{mell}
If $r=1$, then $\mlow= \mup = m$ and $\bXp(G) \simeq \bX_m(  G)$.
\end{corollary}
\begin{proof}
If $\mlow =1$, then note that $h_1=0$ by Corollary \ref{hn=0}. Since $h_n$ is a weakly decreasing
function in $n$, we have $h_0 = h_1= h_2 = \cdots = 0$. Therefore the rank of
$\bI_n$ is $0$ for all $n \geq 0$ and hence $\mup =1$.

If $\mlow=2$, then note that $h_1=1$ since $\Delta:\Hom(N^1,\Ga) \map 
\Ext(G,\Ga)$ is injective. But since $r=1$, we have
$h_2=h_3= \cdots = 0$. Therefore $\rk \bI_i$ is constant for all $i\geq 1$
and hence $\mup-1=1$ that is $\mup=2$.
\end{proof}

\section{Kernel of differential characters}
\label{Kernel-Scheme}
\cblue
The main aim of this section is to prove our Theorem \ref{K(G)}. The precise 
statement proved is Theorem \ref{K(G)-2}.
Before we do that, we need to set up the objects and the morphisms that we 
will be required to consider.

\cblack

Let $\mathbbm{B}=\{\Theta_1,\cdots,\Theta_g\} \subset \bX_{\mup}(G)$ be such 
that their images form an
$K$-basis of $\bXp(G)$ where the strict orders of $\Theta_i$s are $o_i$s respectively and $m=m_u:=\max~\{o_i|~i=1,\cdots, g\}.$ Without loss of generality, we may assume $o_1\leq o_2\leq\cdots\leq o_g=m.$ Consider the set $S=\{\tilde{\Theta}_1,\cdots,\tilde{\Theta}_g\}$ where $\tilde{\Theta}_i=\partial^{m-o_i}\Theta_i$ and define the map
$$[\tilde{\Theta}]_m:J^mG\rightarrow V:=\Ga^g$$ as $$[\tilde{\Theta}]_m(\bx):=(\tilde{\Theta}_1(\bx),\cdots, \tilde{\Theta}_g(\bx)).$$
Consider the induced map of prolongation sequences,
\begin{equation}
\label{char}
 \xymatrix{
 J^{m+2}G \ar[rr]^{[\tilde{\Theta}]_{m+2}}\ar[d]& & J^2V\ar[d]\\
 J^{m+1}G \ar[rr]^{[\tilde{\Theta}]_{m+1}}\ar[d] & &JV\ar[d]\\
 J^m G \ar[rr]^{[\tilde{\Theta}]_m}& & V
 }
\end{equation} 
\subsection{Coordinate representation of differential characters}
We will now discuss the representation of our differential characters in terms
of the local \'{e}tale coordinate functions $\bx, \Del(\bx),\dots , \Del^n(\bx)$.
Let 
$$\Theta_i=A_{io_i}\Del^{o_i}(\bx)+\cdots+A_{i1}\Del(\bx)+A_{i0}\bx+f.$$ 
for all $i=1,\dots g$.
Then by Lemma \ref{1} we have 
$$L(\tilde{\Theta}_i):=L(\partial^{n-o_i}\Theta_i)= A_{io_i}.\Del^m(\bx).$$ 

\begin{lemma}
\label{prim_mat}
Let $\bb{B}= \{\Theta_1,\cdots, \Theta_g\}$ be as above such that 
$L(\Theta_i)=A_{io_i}\Del^{o_i}\bx$ for all $i=1,\cdots, g$. Then the row vectors
$\{A_{io_i}\}_{i=1}^g$ are $K$-linear independent. In other words 
$A:=(A_{1o_1}\cdots A_{go_g})^t\in Mat_{g\times g}(K)$ is invertible.
\end{lemma}

\begin{proof}
Since $\bb{B}$ forms a primitive subset of $\bX_\infty(G)$, by Proposition 
\ref{3}, the subset $\{\tilde{\Theta}_1,\dots ,\tilde{\Theta}_g\}$ is 
$K$-linearly independent inside $\bX_m(G)$. Since for all $i=1, \dots , g$,
$L(\tilde{\Theta}_i) = A_{io_i} \Del^m(\bx)$, we must have that 
$\{A_{io_i}\}_{i=1}^g$ are $K$-linearly independent and hence we obtain our 
desired result.
\end{proof}

For all $n$, let $H^nG:= \mb{Ker}(u:J^nG\rightarrow J^{n-1}G)$ satisfying 
the following short exact sequence of group schemes
$$0\rightarrow H^nG\xrightarrow{\iota}J^nG\xrightarrow{u} J^{n-1}G\rightarrow 0.$$
Then note that $H^nG$ is a vectorial group isomorphic to $\hG^g$.
\begin{proposition}
\label{leadterm}
Given $\Theta\in\bX_n(G),$ such that the strict order of $\Theta$ is $n$, consider $\iota^*\Theta\in \Hom(H^n G,\Ga).$ Then $\iota^*\Theta=L(\Theta)$.
\end{proposition}
\begin{proof}
In terms of local coordinate functions we have
$$\Theta=A_n\Del^{n}(\bx)+A_{n-1}\Del^{n-1}(\bx)+\cdots +A_0\bx+f,$$ where $f\in(\bx)$. Then $$\iota^*\Theta=\Theta(\Del^n(\bx),0,\cdots,0)=A_n\Del^n(\bx)$$ and we are done.
\end{proof}

Consider for all $n\geq m$ the map in equation (\ref{char}) 
$$[\tilde{\Theta}]_n:J^nG\longrightarrow J^{n-m}V.$$ 
is given by
$$[\tilde{\Theta}]_n=(u^*\tTheta_1,\cdots,u^*\tTheta_g,\cdots,\partial^{n-m}\tTheta_1,\cdots,\partial^{n-m}\tTheta_g).$$

Consider the following diagram of group schemes
\begin{equation}
\label{Kernels}
\xymatrix{
 0\ar[d] & & 0\ar[d] & \\
H^{n+1}G\ar[rr]^{\iota^*[\tTheta]_{n+1}}\ar[d] & &H^{n+1-m } V\ar[d] & \\
J^{n+1}G\ar[rr]^{[\tTheta]_{n+1}}\ar[d]_u & & J^{n+1-m}V\ar[r]\ar[d]_u & 0\\
J^nG\ar[rr]^{[\tTheta]_{n}}\ar[d] &  & J^{n-m} V\ar[r]\ar[d]& 0\\
0 & & 0 & \\
}
\end{equation}
where the map $\iota^*[\tTheta]_n: H^{n+1}G \map J^{n+1-m}V$ naturally factors 
through $H^{n+1-m}V$ as shown above for all $n+1 \geq m$. 

\begin{proposition}
\label{Char_Matrix}
\begin{enumerate}
\item For all $n\geq m$, $$\iota^{*}[\tTheta]_n:H^nG \simeq \Ga^g\longrightarrow H^{n-m}V\simeq \Ga^g $$ is given by $$\iota^*[\tTheta]_n=A,$$ where 
$A:=(A_{1o_1}\cdots A_{go_g})^t.$
\item By Lemma \ref{prim_mat}, $A$ is invertible, that implies $$\iota^*[\tTheta]_n:H^nG\longrightarrow H^{n-m}V$$ is an isomorphism.
In particular, $\iota^* [\tTheta]_m:N^mG \map V$ is a surjection.

\item The map $[\tTheta]_n:J^nG\longrightarrow J^{n-m}V$ is surjective.
\end{enumerate}
\end{proposition}
\begin{proof}
For all $n \geq m$, the pull-back map $\iota^*[\tTheta]_n: H^nG \map H^{n-m}V$
is given by 
\beqar
\iota^* [\tTheta]_n &=& (L(\partial^{n-m}\tTheta_1), \cdots , 
L(\partial^{n-m}\tTheta_g)) \\
& = & A. \Del^n(\bx) 
\eeqar
where $A = (A_{1o_1} \cdots A_{go_g})^t$ and this proves $(1)$.

By Lemma \ref{prim_mat}, $A$ is an invertible matrix and hence this proves
the first part of $(2)$. We have $H^mG \inj N^mG$ and since $\iota^*[\tTheta]_m
:H^mG \map V$ is an isomorphism implies that $\iota^*[\tTheta]_m: N^mG \map
V$ is a surjection. 

We will show $(3)$ by induction on $n$. For $n = m$, consider the map 
$\iota^*[\tTheta]_m: H^m G \map V$. By $(2)$ this morphism is invertible and
therefore implies that $[\tTheta]_m: J^mG \map V$ is a surjection. Now suppose
the result is true for $n$. Then in diagram (\ref{Kernels}), by $(2)$,
$\iota^*[\tTheta]_n$ is an isomorphism and $[\tTheta]_n$ is surjective by our
induction hypothesis. Hence this implies that $[\tTheta]_n$ is surjective
and we are done.
\end{proof}

By Proposition \ref{Char_Matrix}, we can refine our diagram (\ref{Kernels})
to the following 
\begin{equation}
\label{Kernels2}
\xymatrix{
&  & 0\ar[d] & & 0\ar[d] & \\
& 0 \ar[r]\ar[d] & H^{n+1}G\ar[rr]_\sim^{\iota^*[\tTheta]_{n+1}}\ar[d] & &H^{n+1-m } V\ar[d]  & \\
0\ar[r] & K^{n+1}G\ar[r]\ar[d]_{u} & J^{n+1}G\ar[rr]^{[\tTheta]_{n+1}}\ar[d]_u & & J^{n+1-m}V\ar[r]\ar[d]_u & 0\\
0\ar[r] & K^nG\ar[r]\ar[d] & J^nG\ar[rr]^{[\tTheta]_{n}}\ar[d] &  & J^{n-m} V\ar[r]\ar[d]& 0\\
& \text{Coker}_1 & 0 & & 0 & \\
}
\end{equation}
where for all $n \geq m$, $K^nG := \mb{Ker} ([\tTheta]_n)$ .
\cblue
Note that for all $n$, $K^nG$ is independent of the choice of the basis 
$\bb{B}=\{\Theta_1, \dots ,\Theta_g\}$ of $\bXp(G)$.

\color{black}

\begin{theorem}\label{K_nG}
For all $n\geq m$, we have $$K^{n+1}G\simeq K^nG.$$
\end{theorem}
\begin{proof}
By Proposition (\ref{Char_Matrix})(2), $\iota^*[\tTheta]_n$ is an isomorphism. 
Hence by snake lemma, this implies that Coker$_1=\{0\}$ and this proves our
result.
\end{proof}

Define
$J^\infty G:= \varprojlim_u J^nG$, $K^\infty G:= \varprojlim_u K^nG$ and
$J^\infty V:= \varprojlim_u J^nV\simeq \Ga^\infty$.
The above limits exist because the projection maps $u$ are affine morphisms
for all levels $n$.
Also passing to the limit in diagram (\ref{Kernels}) 
we have \begin{equation}
\label{prim_exact}
0\longrightarrow K^\infty G\rightarrow J^\infty G\xrightarrow{[\tTheta]_{\infty}}J^\infty V\rightarrow 0,
\end{equation}
where $[\tTheta]_{\infty}:J^\infty G\longrightarrow J^\infty V$ is a map of 
$D$-group schemes. Hence naturally, $K^\infty G=\Ker[\tTheta]_{\infty}$ obtains a $D$-group scheme structure.

\begin{corollary}
\label{Cor_1}
We have 
\begin{enumerate}
\item $K^\infty G\simeq K^mG.$
\item $\dim K^\infty G=\dim K^mG=mg.$
\end{enumerate}
In other words, we have $K(G) := K^mG$ to be a smooth commutative $D$-group 
scheme of relative dimension $mg$ over $S$.

\end{corollary} 
\begin{proof}
\begin{enumerate}
\item This is immediate from Theorem \ref{K_nG}.
\item Note that the morphism $$[\tTheta]_m:J^mG\longrightarrow V$$ is a surjection since $\iota^*[\tTheta]_m:H^mG\rightarrow V$ is an isomorphism of vector groups by Proposition \ref{Char_Matrix}(2). Hence the result follows since $\dim J^mG=(m+1)g$ and $\dim V=g.$ 
\end{enumerate}
\end{proof}

\begin{theorem}
We have $$0\longrightarrow K(G)\longrightarrow J^{\infty}G\xrightarrow{[\tTheta]_\infty} J^\infty V\longrightarrow 0. $$
\end{theorem}
\begin{proof}
The result follows immediately from equation (\ref{prim_exact}) and 
Corollary \ref{Cor_1}.
\end{proof}

Consider the composition 
$$N^m G \stk{\iota}{\longrightarrow} J^m G 
\stk{[\tTheta]_m} {\longrightarrow} V$$
and define
$L(G) := \ker ([\tTheta]_m \circ \iota : N^mG \map V)$. Hence $\iota$ induces
a map, also denoted by $\iota : L(G) \map K(G)$ as shown in the dotted arrow
below
\begin{align}
\label{KG}
\xymatrix{
0 \ar[r] & L(G) \ar[r]\ar@{.>}[d]^\iota & N^{m}G\ar[rr]^-{[\tTheta]_{m}
\circ \iota} \ar[d]^-\iota & & V\ar[d]^{\mathbbm{1}} \ar[r] & 0 \\
0\ar[r] & K(G)\ar[r] & J^{m}G\ar[rr]^-{[\tTheta]_{m}} & & 
V\ar[r] & 0.\\
} \end{align}

Note that the map $[\tTheta]_m \circ \iota : N^m G \map V$ is a surjection
because of Proposition \ref{Char_Matrix} (2).

\begin{theorem}
\label{K(G)-2}
The group scheme $K(G)$ is a finite dimensional $D$-group scheme 
which is a vectorial extension of $G$ as follows
$$
0 \longrightarrow L(G) \stk{\iota}{\longrightarrow} K(G) 
\longrightarrow G \longrightarrow 0
$$
where $L(G) := \ker (\iota^*[\tTheta]_m: N^mG \map V)$ is a vector group of 
relative dimension $(m-1)g$ over $S$.
\end{theorem}

\begin{proof}
Consider the following diagram of group schemes which is an extension of
the diagram (\ref{KG})
\begin{equation}
\label{D-module}
\xymatrix{
& \mb{ker} \ar[d] & 0\ar[d] & & 0\ar[d] & \\
0 \ar[r] & L(G) \ar[r]\ar[d]^\iota & N^{m}G\ar[rr]^-{\iota^*[\tTheta]_{m}}
\ar[d]^-\iota & & V\ar[d]^{\mathbbm{1}} \ar[r] & 0 \\
0\ar[r] & K(G)\ar[r]\ar[d] & J^{m}G\ar[rr]^-{[\tTheta]_{m}}\ar[d]_u & & 
V\ar[r]\ar[d] & 0\\
& \mb{coker}  & G  &  & 0. & \\
}
\end{equation}

Then  $L(G) \simeq \Ga^{(m-1)g}$ since $[\tTheta]_m \circ \iota : N^mG 
\map V$ is a surjection. 
And by snake lemma we obtain that  $\mb{ker} \simeq \{0\}$ and 
$\mb{coker} \simeq  G$. Hence the result follows considering the first column
of the above diagram.
\end{proof}

\subsection{Concluding remarks}
We now make a few remarks about the relation of $K(A)$ and the universal
vectorial extension of $A$ when $A$ is an abelian 
scheme over $S$. Consider the universal vectorial
extension $E(A)$ satisfying 
\begin{align}
\label{univ-extn}
0 \map H^0(A,\Omega_A) \map E(A) \map A \map 0,
\end{align}
as in equation $(5.4)$ in \cite{Bost}. By the universal property of 
(\ref{univ-extn}), there exists a unique map $f: H^0(A,\Omega_A) \map L(A)$ 
of vector groups such that our exact sequence in Theorem \ref{K(G)} is a 
push out of (\ref{univ-extn}) by $f$ as follows
\begin{align}
\xymatrix{
0 \ar[r] & H^0(A,\Omega_A) \ar[d]^f \ar[r] & E(A) \ar[d]^-\Psi \ar[r] & A 
\ar@{=}[d] \ar[r] & 0 \\
0 \ar[r] & L(A) \ar[r] & K(A) \ar[r] & A \ar[r] & 0.
}
\end{align}
Note that $E(A)$ is also naturally a $D$-group scheme. Therefore it appears
pertinent, at the moment, to conjecture that the map $\Psi$ is a map of 
$D$-group schemes for abelian schemes $A$. This is a deep question and
we relegate further investigation of it to future work.

\textbf{Acknowledgments.} 
The authors are grateful to the
anonymous referee for offering  perceptive suggestions from which
this paper has greatly benefitted.
The first author would like to thank CSIR for  financial support under the scheme 09/1031(0011)/2021-EMR-I.
The second author would like to thank James Borger for many insightful 
discussions. The second author was partially supported by the SERB grant
SRG/2020/002248. The authors also wish to thank Sudip Pandit for helpful remarks.

\end{document}